\documentclass[reqno]{amsart}

\usepackage{amsmath}
\usepackage{amssymb}
\usepackage{amsthm}
\usepackage{amsfonts}
\usepackage{dsfont}
\usepackage{color}
\usepackage{enumerate}
\usepackage[margin=38mm]{geometry}

\usepackage{braket}

\newcommand{\complex}{\mathbb{C}}

\newcommand{\paraa}[1]{\big(#1\big)}
\newcommand{\parab}[1]{\Big(#1\Big)}

\newtheorem{theorem}{Theorem}[section]

\newtheorem{lemma}[theorem]{Lemma}
\newtheorem{proposition}[theorem]{Proposition}

\theoremstyle{definition}
\newtheorem{definition}[theorem]{Definition}
\theoremstyle{remark}

\numberwithin{equation}{section}

\newcommand{\zb}{\overline{z}}
\newcommand{\xb}{\overline{x}}
\newcommand{\yb}{\overline{y}}
\newcommand{\wb}{\overline{w}}
\newcommand{\ib}{\overline{i}}
\newcommand{\jb}{\overline{j}}

\title{Weighted Bergman kernels and $\star$-products}

\author{Andreas Sykora}
\address[Andreas Sykora]{}
\email{syko@gelbes-sofa.de}

\thanks{}

\subjclass[2000]{}

\begin{document}

\begin{abstract}
  We calculate the weighted Bergman kernel on
  a complex domain with a weight of the form $\rho=e^{-\alpha\phi}\mu g$,
  where $\alpha$ is a positive real number, $\phi$ is a Kähler potential, g is the determinant
  of the corresponding Kähler metric and $\mu$ is a real-valued positive function.
  Several $\star$-products related to the Bergman kernel are determined. Explicit formulas are
  provided up to first order.
\end{abstract}

\maketitle

%\tableofcontents

\section{Introduction}

Let $\Omega$ be a domain in $\complex^N$ ($N\geq1)$) and
$\rho$ a positive smooth weight function on $\Omega$.
It is known that the \emph{weighted Bergman space} $L^2_\text{hol}(\Omega,\rho)$ of
all holomorphic functions in $L^2(\Omega,\rho)$ has a
reproducing kernel $K(x,\yb)$, i.e.
\begin{align*}
  f(x) = \int_\Omega  K(x,\yb) f(y) \, \rho(y,\yb) d^{2N} y
\end{align*}
for $f\in L^\infty(\Omega)$, which is called
\emph{weighted Bergman kernel}.
Abusing notation, we will often write $f(x)$, $f(\xb)$ or $f(x,\yb)$
and mean a function, which is holomorphic or anti-holomorphic in the
respective argument (and not its value at the points $x,y$).

The Segal-Bargmann space is the prototypic example for
a Bergman space, where $\Omega=\complex^N$, the 
weight function is $\rho(x,\xb)=e^{-\alpha x\xb}$ and the
Bergman kernel is $K(x,\yb)=\frac{\alpha^N}{\pi^N}e^{\alpha x\yb}$.
See, for example \cite{Bar20}, which is a review on weighted Bergman kernels used in
mathematical physics.
When $\Omega$ is the unit disc and
the weight function is $\rho(x,\xb)=\frac{\alpha+1}{\pi}(1-|x|^2)^\alpha$
where $\alpha>-1$ is a parameter, the 
Bergman kernel is $K(x,\yb)=(1-x\yb)^{-\alpha-2}$
Numerous explicit examples of Bergman kernels are known (\cite{Bon20,Boa97,Den21}).

Relaxing the condition of $\Omega$ being a domain,
one can consider Kähler manifolds $M$ with
a Hermitian line bundle $L\rightarrow M$
having a curvature that equals the symplectic form
$\omega$ of the Kähler manifold. Vice versa, a Kähler manifold
is said to be quantizable, when its symplectic form has integer cohomology and 
such an Hermitian line bundle exists. The Bergman
space is then the space of holomorphic sections $L^2_\text{hol}(M,L)$,
which is a subspace of $L^2(M,L)$, the space of square integrable sections with
respect to the Hermitian inner product \cite{Cha03, Del21, Hez21}.
In particular, one considers the sequence of Bergman spaces
$L^2_\text{hol}(M,L^n)$ of the powers 
of the line bundle in the limit $n\rightarrow\infty$.
In this context, the weight has locally the form $\rho=e^{-n\phi}g$,
where $\phi$ is a local Kähler potential and $g=\omega^N / N!$
is the determinant of the Kähler metric.

For $n\rightarrow \infty$ asymptotic formulas 
for the Bergman kernel $K_n(x,\xb)$ of the
the line bundle $L^n$ were firstly derived in
%Tian, Yau, Zelditch, Lu and Catlin
\cite{Tia90,Zel98,Cat99,Lu00},
and extended to the off-diagonal in \cite{Cha03, Del21, Hez21}.
The asymptotic expansion has the form
\begin{align} \label{kern1}
  K_n(x,\yb) = \frac{n^N}{\pi^N} e^{n\phi(x,\yb)} \sum_{m=0}^\infty \frac{1}{n^m} k_m(x,\yb)
\end{align}
where the functions $k_m(x,\yb)$ depend on the respective $\phi$.
In \cite{Dou10} and \cite{Kle09}, this asymptotic expansion 
is derived with the aid of a path integral and the density matrix
projected on the lowest Landau level, respectively.

Returning to the Kähler manifold being a domain
$\Omega$ in $\complex^N$, it is shown in \cite{Eng00-2} that for the weight $\rho=e^{-\alpha\phi}g$
with $\alpha$ a positive real number and $g$ the determinant of the Kähler metric for the potential $\phi$,
the Bergman kernel $K_\alpha(x,\yb)$ has
asymptotic expansion
\begin{align} \label{kern2} 
  K_\alpha(x,\yb) = \frac{\alpha^N}{\pi^N} e^{\alpha\phi(x,\yb)} \sum_{m=0}^\infty \frac{1}{\alpha^m} k_m(x,\yb)
\end{align}
In \cite{Eng00-1,Eng-??} it is shown that for the weight
$\rho=e^{-n\psi-\phi}$, where $\psi$ and $\phi$ are
plurisubharmonic functions and $k$ is a natural number, the Bergman kernel $K_n(x,\yb)$ has
asymptotic expansion
\begin{align} \label{kern3} 
  K_n(x,\yb) = \frac{n^N}{\pi^N} e^{n\phi(x,\yb)+\psi(x,\yb)} \sum_{m=0}^\infty \frac{1}{n^m} k_m(x,\yb)
\end{align} 
with functions $k_m(x,\yb)$ depending on $\phi$ and $\psi$.

For studying the Bergman space, in \cite{Cha03} Töplitz operators
generalizing the kernel (\ref{kern1}) have been introduced. These kernels are
called covariant Töplitz operator kernels in \cite{Del21}. For a possibly solely
formal sum
\begin{align} 
  f(x,\yb) = \sum_{m=0}^\infty \frac{1}{n^m} f_m(x,\yb)
\end{align} 
where the $f_m(x,\yb)$ are  functions on $M\times M$,
the covariant Töplitz operator kernel is
\begin{align} \label{cov_toep} 
  T_f(x,\yb) = \frac{n^N}{\pi^N} e^{n\phi(x,\yb)} \sum_{m=0}^\infty \frac{1}{n^m} f_m(x,\yb)
\end{align} 
Under further assumptions, it is shown in in \cite{Cha03} and \cite{Del21} that
the product of two covariant Töplitz operators is again a covariant Töplitz operator.
The Bergman kernel (\ref{kern1}) becomes a covariant Toeplitz operator
with symbol $k=\sum k_n$.

Starting point of the author for the present work 
was a physical problem, relating to the quantization
of a system on a domain $\Omega$ in $\complex^N$
with weight $\rho= e^{-\alpha\phi}$ (without the factor $g$), which
for $\alpha$ being a natural number is a special case of (\ref{kern3}).
In particular, when the term of first order of the asymptotic expansion
of the Bergman kernel is known, the semi-classical limit of the quantized system,
which is mainly encoded in its Poisson bracket, can be derived.
However, only for the asymptotic expansions (\ref{kern1}) and
(\ref{kern2}) explicit formulas were provided for the first few $k_m$.
These expressions can be derived, from recurrence relations for
the functions $k_m$, allowing to explicitly compote the function $k_m$ from
the functions $k_0,\dots,k_{m-1}$, see \cite{Eng00-2, Cha03, Hez21}, for example.

In the following, we will define covariant Töplitz operators for
more general weights. In particular, as standing assumptions,
$\Omega \subset \complex^N$ is a domain,
and we will consider weights of the form
\begin{align}  \label{rho_weight}
  \rho(x,\xb)= e^{-\alpha\phi(x,\xb)} \mu(x,\xb) g(x,\xb)
\end{align}
where $\phi(x,\xb)$ is a Kähler potential, i.e. a strictly plurisubharmonic function,
and $\mu(x,\xb)$ is a real-valued positive function. We further assume
that $\phi$ and $\mu$ are real-analytic and have a holomorphic extension
to a neighbourhood of the diagonal of $\Omega\times\Omega$.
$g$ is the determinant of the Kähler metric $g_{i\jb}=\partial_i\partial_{\jb}\phi$
associated with $\phi$. $\alpha$ is a positive real number.

Based on the weight (\ref{rho_weight}), it shows that it is beneficial
to define the kernel of the covariant Töplitz operator $T_f$
of the \emph{symbol} $f$ (see section 2) to be
\begin{align*} 
  T_{f}(x,\yb) = \frac{\alpha^N e^{\alpha\phi(x,\yb)}}{\pi^N \mu(x,\yb)} f(x,\yb)
\end{align*}
where we have included the function $\mu$ in the denominator.

In general, a "symbol" can be considered as a function with good properties,
such that the Töplitz operator $T_f$ exists.
Let us assume that $f_1$, $f_2$, $f_3$ are symbols and
that the product operator $T_{f_1}M_{f_2}T_{f_3}$ is also
a covariant Töplitz operator. $M_{f_2}$ denotes the multiplication operator with the
symbol $f_2$ from the left. This means that there
is a \emph{triple symbol} $S(f_1,f_2,f_3)$ with
\begin{align*} 
  T_{S(f_1,f_2,f_3)} = T_{f_1}M_{f_2}T_{f_3}
\end{align*} 
Let $f$ and $h$ be symbols, then the
\emph{Berezin-Töplitz-$\star$-product} of them 
is  
\begin{align} \label{ber_top_star}
  \paraa{f_1\star f_2}(x,\xb) = S\paraa{f_1,1,f_2}(x,\xb)
\end{align}  
 It then follows that for the corresponding Töplitz operators
\begin{align*} 
  T_{f} T_{h} = T_{S(,f, 1 , h)} = T_{f \star h}
\end{align*} 
Since the operator product is associative, the same applies
to the $\star$-product.
The \emph{symbol $k$ of the Bergman kernel} is the unit
of the Berezin-Töplitz-$\star$-product. When one assumes that it
exists, the corresponding Toeplitz operator $T_k$, i.e. the
\emph{Bergman kernel} is a projector
\begin{align*}  
  T_k T_f = T_{k\star f} = T_{f\star k} = T_f T_k = T_f
\end{align*}
The corresponding subspace onto which $T_k$ projects is
the Bergman space.

The $\star$-product (\ref{ber_top_star}) has the disadvantage that
instead of the function $1$, the symbol $k$ of the Bergman kernel
is its unit. For quantization, a $\star$-product with the function $1$ as unit is more
useful. To this end, Berezin \cite{Ber74} introduced the Berezin $\star$-product.
This is a $\star$-product equivalent to the Berezin-Töplitz-$\star$-product
(\ref{ber_top_star}) by the Berezin transform. 
In particular, the \emph{contravariant symbol} $\psi(f)$ of the symbol $f$ is  
defined to be
\begin{align}  \label{ber_trans}
  \psi(f) = S(k,f,k)
\end{align}  
where $k$ is the symbol of the Bergman kernel and $S$ is the triple
symbol from above.
$\psi$ is called \emph{Berezin transform}.
Due to $\psi(1)=S(k,1,k) = k\star k = k$, the symbol $k$
of the Bergman kernel is the contravariant symbol of $1$.
When one assumes that the contravariant symbol 
exists, the corresponding Toeplitz operator is 
\begin{align*}  
  T_{\psi(f)}=  T_{S(k, f, k)} = T_k M_f T_k
\end{align*}
This means that the contravariant symbol is the projection of
the multiplication operator on the Bergman space.
The Berezin $\star$-product or contravariant $\star$-product is defined by
\begin{align} \label{con_star} 
  f\star_{\text{con}} h = \psi^{-1} \paraa{ \psi(f) \star \psi (h) }
\end{align}  
and has the function $1$ as unit, since $\psi(1)=k$.

A further possibility to define a $\star$-product with $1$ as unit starts with
the covariant symbol of a symbol $f$, which is 
\begin{align*}  
  \varphi(f) = \frac{f}{k}
\end{align*}
where $k$ is the symbol of the Bergman kernel.
The covariant $\star$-product is defined by
\begin{align}  \label{cov_star} 
  f\star_{\text{cov}} h = \varphi \paraa{ \varphi^{-1}(f) \star \varphi^{-1}(h) }
   = \frac{1}{k}\paraa{(k f)\star(k g)}
\end{align}    
which also has the function $1$ as unit.

All the relations above can be derived, when one knows the
triple symbol $S(f_1,f_2,f_3)$. In the following, we will define formal
symbols and a formal triple symbol, which are power series
in a formal parameter $\hbar$, which can be identified
with $\frac{1}{\alpha}$. With the formal triple symbol,
formal $\star$-products and a formal Bergman kernel can
be defined. In particular in section 2, we will use
analytic properties of the triple symbol to show that it fulfils a 
generalized associativity law and in section 3, we will use the
generalized associativity law to show the following theorem.

\begin{theorem} \label{main_theo}
  Let $\Omega\subset \complex^N$ be a domain and $\rho$ a
  weight of the form (\ref{rho_weight}). Let $R_n$ be the differential
  operators defined by (\ref{int_exp}) occurring in the asymptotic
  expansion of the integral (\ref{J_int}).
  Furthermore, let $f$ and $h$ be formal symbols (see Definition \ref{sym_def}).
  Then the formal Berezin-Töplitz-$\star$-product is
  \begin{align} \label{ber_top_exp}
   \paraa{f\star h}(x,\xb)=  \sum_{n=0}^{\infty} \hbar^n
     R_n\paraa{ f^{x\yb} h^{y\xb} \tilde{\mu}^{x\xb y\yb} }(x,\xb) 
  \end{align}
  ($\tilde{\mu}$ is a function depending on $\mu$ and is defined in (\ref{dia_fun_mu}).
  (\ref{ber_top_exp}) is an associative product having the formal
  symbol $k$ of the Bergman kernel as unit.
  
  Up to first order, the formal Berezin-Töplitz $\star$-product is
  \begin{align} \label{ber_top_star_h2}
     f\star h & =f h  + \hbar \paraa { g^{i\jb} f_{,\jb} h_{,i}  + f h (\Delta \mu + \frac{1}{2} R) }  +  \mathcal{O}(\hbar^2)    
  \end{align}
   where $g^{i\jb}$ is the inverse of the Kähler metic induced by $\phi$.
  The formal symbol $k$ of the Bergman kernel is
  \begin{align} \label{ker_sym_h2}
     k = 1 - \hbar\paraa{ \Delta \mu + \frac{1}{2} R } + \mathcal{O}(\hbar^2)  
  \end{align}
  The contravariant $\star$-product is
  \begin{align} \label{con_star_h2}
   f\star_{\text{con}} h  = fh - \hbar g^{i\jb} f_{,i} h_{,\jb} +  \mathcal{O}(\hbar^2)    
  \end{align}
  The covariant $\star$-product is
  \begin{align} \label{cov_star_h2}
    f\star_{\text{cov}} h  = fh + \hbar \paraa { g^{i\jb} f_{,\jb} h_{,i}  }  +  \mathcal{O}(\hbar^2)    
  \end{align}
  For two functions $f,h$, the Poisson structure of the classical limit is
  \begin{align} 
    \{f, h\}  =  g^{i\jb} \paraa { f_{,\jb} h_{,i}  - f_{,i} h_{,\jb}  }
  \end{align}
\end{theorem}
The Poisson structure of the classical limit is the first order of the commutator
of the $\star$-products defined above. It can be shown that equivalent
$\star$-products have the same Poisson structure.
 
Note that $\mu$ is not present up to first order in the 
contravariant $\star$-product and the  covariant $\star$-product,
and therefore in the Poisson structure of the classical limit. This means
that the quantization of the Kähler manifold $\Omega$ with the weight
$\rho=e^{-\alpha\phi}\mu g$ is independent of the function $\mu$.

\section{Töplitz operators}

In this section, we define formal symbols, which are formal
power series of functions and analytic symbols, which are functions
having an asymptotic expansion in $\frac{1}{\alpha}$. The parameter $\hbar$ 
of the formal power series can be identified with $\frac{1}{\alpha}$, when the
corresponding power series converges. 

\begin{definition}  \label{sym_def}
  The formal power series $f\in C^\omega(\Omega)[[\hbar]]$
   \begin{align}
      f(x,\xb) = \sum_{m=0}^{\infty} \hbar^m f_m(x,\xb)
   \end{align}
   is a \emph{formal symbol}, when there
   is a neighbourhood of the diagonal in $\Omega\times\Omega$,
   where the real-analytic functions $f_m(x,\xb)$ have a holomorphic extension $f_m(x,\yb)$.
   The corresponding \emph{analytic symbol of order $M$} is
   the function
   \begin{align}
      f_\alpha^{(M)}(x,\xb) = \sum_{m=0}^{M} \frac{1}{\alpha^m} f_m(x,\xb)
   \end{align}
   When $f_\alpha= \lim_{M\rightarrow\infty} f_\alpha^{(M)}$ exists in a neighbourhood of
   the diagonal, it is called \emph{analytic symbol of infinite order} or simply \emph{analytic symbol}.
   A \emph{principal symbol} is an analytic symbol of order $0$, i.e. it is a function $f$ on
   $\Omega$, which has a holomorphic extension to a neighbourhood
   of the diagonal of $\Omega\times\Omega$.
\end{definition} 
Each analytic symbol of order $M$ has a holomorphic extension in a neighbourhood
of the diagonal. Infinite analytic symbols with an additional convergence property
are defined in \cite{Del21,Cha21}.

\begin{definition}  
The kernel of the \emph{covariant Töplitz operator} of an analytic symbol $f_\alpha$
associated with the weight (\ref{rho_weight}) is
\begin{align}  \label{cov_tmu}
  T_{f}(x,\yb) = \frac{\alpha^N e^{\alpha\phi(x,\yb)}}{\pi^N \mu(x,\yb)} f_\alpha(x,\yb)
\end{align}
\end{definition}  
This is in analogy with (\ref{cov_toep}) and extends to general $\alpha$.
In view of (\ref{kern3}) we have included the function
$\mu$ in the prefactor, which will simplify formulas in the following.

We now consider three principal symbols $f_1$, $f_2$, $f_3$.
Let $M_{f}$ be the multiplication operator for left multiplication with
the principal symbol $f$. The triple product operator defined by
\begin{align}  \label{trip_prod}
  T_{f_1,f_2,f_3} = T_{f_1} M_{f_2} T_{f_3}
\end{align}
has kernel
\begin{align} \label{tri_ker}
    T_{f_1,f_2,f_3}(x,\zb) & = \int T_{f_1}(x,\yb) f_2(y,\yb) T_{f_3}(y,\zb) \rho(y,\yb) \, d^{2N}y \\
    & = \frac{\alpha^{2N} e^{\alpha\phi(x,\zb)}}{\pi^ {2N} \mu(x,\zb)}
               \int  e^{-\alpha\tilde{\phi}(x,\zb,y,\yb) } \tilde{\mu}(x,\zb,y,\yb)  
                  f_1 (x,\yb) f_2 (y,\yb) f_3(y,\zb) \; g(y,\yb) \, d^{2N}y
\end{align}
where
\begin{align}  
  \tilde{\phi}(x,\zb,y,\yb) & = \phi(x,\zb) +\phi(y,\yb) - \phi(x,\yb) - \phi(y,\zb) \label{dia_fun_phi}  \\
  \tilde{\mu}(x,\zb,y,\yb) & = \frac{\mu(x,\zb)\mu(y,\yb)}{\mu(x,\yb)\mu(y,\zb)} \label{dia_fun_mu}
\end{align}
are generalizations of Calabi's diastasis function.

The integral in (\ref{tri_ker}) is of the form
\begin{align}  \label{J_int}
  J_\alpha\paraa{\tilde{\phi}, f}(x,\zb) = \int  e^{-\alpha \tilde{\phi}(x,\zb, y,\yb) } f(x,\zb,y,\yb) g(y,\yb) \, d^{2N}y 
\end{align}
where $f(x,\zb,y,\yb)$ is a real-analytic function,
which exists in a neighbourhood of the diagonal in $\Omega\times\Omega\times\Omega$.
In \cite{Eng00-2} (Theorem 3), \cite{Cha03}, \cite{Del21} (Propositions, 3.12 and 4.3) it is shown that
the integral (\ref{J_int}) has an asymptotic expansion in $\alpha$
in a neighbourhood of the diagonal $\Omega\times\Omega$.
In \cite{Eng00-2} this asymptotic expansion is explicitly determined 
\begin{align} \label{int_exp}
   J_\alpha\paraa{\tilde{\phi}, f}(x, \zb) = \frac{\alpha^N}{\pi^N} \sum_{n=0}^{\infty} \frac{1}{\alpha^n} R_n\paraa{f}(x,\zb)
\end{align}
where $R_n$ are differential operators of order $2n$.
In general, the function $f$ depends on $x$, $\zb$, $y$ and $\yb$, and 
the operators $R_n$ contain at least $2n$ partial derivatives $\partial_i=\frac{\partial}{\partial y^i}$ and
$\partial_{\ib}=\frac{\partial}{\partial \yb^i}$ with respect to $y$ and $\yb$. After application of 
these partial derivatives, the result is evaluated at $x=y$ and $\zb=\yb$.
Below we will give explicit formulas for $R_0$ and $R_1$, see (\ref{R0_eq}) and (\ref{R1_eq}).

Applying the asymptotic expansion (\ref{int_exp})
to the kernel (\ref{tri_ker}) results in
\begin{align}  \label{trip_exp}
   T_{f_1,f_2,f_3}(x,\zb)
     = \frac{\alpha^{N} e^{\alpha\phi(x,\zb)}}{\pi^{N} \mu(x,\zb)}
		\sum_{n=0}^{\infty}
		  \frac{1}{\alpha^{n}} R_n\paraa{  f_1^{x\yb} f_2^{y\yb}f_3^{y\zb} \tilde{\mu}^{x\zb y\yb} } (x,\zb) 
\end{align}
where we have abbreviated the arguments of the functions inside the $R_n$
as upper indices to shorten notation. ( \ref{trip_exp}) is a covariant
Töplitz operator, with an analytic symbol of infinite order.

\begin{proposition}
    Let $f_1$, $f_2$ and $f_3$ be principal symbols.
    Then  the triple operator (\ref{trip_prod}) is a Töplitz operator
    wih analytic symbol
   \begin{align} \label{tri_sym}
       S_\alpha\paraa{f_1,f_2,f_3} (x,\zb) 
        = \sum_{n=0}^{\infty} \frac{1}{\alpha^n} R_n\paraa{ f_1^{x\yb} f_2^{y\yb}f_3^{y\zb} \tilde{\mu}^{x\zb y\yb} } (x,\zb) 
  \end{align}
\end{proposition}

The triple symbol (\ref{tri_sym}) fulfils a generalized associativity law,
which we will later use to define an associative product for formal symbols.
To shows this, we need 
\begin{lemma} \label{cphi_lem}
   The function
   \begin{align*} 
   \check{\phi}(x,\zb, y,\yb,w,\wb)
      & = \tilde{\phi}(x, \zb,  y, \yb) + \tilde{\phi}(y, \zb,  w, \wb)
      = \tilde{\phi}(x, \zb,  w, \wb) + \tilde{\phi}(x, \wb,  y, \yb) \\
      & = \phi(x, \zb) + \phi(y, \yb) + \phi(w, \wb) - \phi(x,\yb)- \phi(y,\wb)-\phi(w,\zb)
  \end{align*}  
  is an analytic phase, such as defined in \cite{Del21}, 3.11. Thus, for every
  real-analytic function $f$, the integral
  \begin{align} \label{cJ_int}
    J_\alpha\paraa{\check{\phi}, f}(x,\zb) =  \int_{\Omega\time\Omega}
         e^{-\alpha \check{\phi}(x,\zb, y,\yb,w\wb) } f(x,\zb,y,\yb,w,\wb) g(y,\yb) \, d^{2N}y  \, d^{2N}w 
  \end{align}  
  exists for $(x,\zb)$ in a neighbourhood of the diagonal of $\Omega\times\Omega$
  and is a real-analytic function. which has an asymptotic expansion in $\frac{1}{\alpha}$.
\end{lemma}
\begin{proof}
  Since $\phi$ has a holomorphic extension on a neighbourhood
  of the diagonal, the same applies to $\check{\phi}$ and the function
    \begin{align*} 
     \Phi_\lambda (y_1,\yb_2, w_1, \wb_2) = \check{\phi}(x,\zb,y_1,\yb_2, w_1,\wb_2)
   \end{align*}
   with $\lambda=(x,\zb)$ is such as in \cite{Del21}, 3.11. In particular
   with $X_\lambda=(x, \zb, x,\zb)$
   \begin{align*}
      X_{\lambda=0} & = 0, \qquad \Phi_\lambda (X_\lambda) = 0 \\
      \partial_I \Phi_\lambda (X_\lambda) & = 0, \qquad
      \partial_{IJ} \Phi_\lambda (X_\lambda) \text{ is positive definite}
   \end{align*}
   where $\partial_I = \partial_{y_1}, \partial_{\yb_2}, \partial_{w_1},\partial_{\wb_2}$.
   The existence of the integral (\ref{cJ_int}) follows from the complex stationary
   phase lemma, see \cite{Del21}, Proposition 3.12.
\end{proof}

\begin{proposition} \label{5_sym_prop}
    Let $f_1$, $f_2$, $f_3$, $f_4$ and $f_5$ be principal symbols.
    Then the operator
     $\hat{T} = T_{f_1}M_{f_2} T_{f_3} M_{f_4} T_{f_5}$
    is a Töplitz operator with analytic symbol
   \begin{align} \label{5_sym}
       S_\alpha\paraa{f_1,f_2,f_3, f_4, f_5} 
        = S_\alpha\paraa{ f_1,f_2, S_\alpha(f_3, f_4, f_5) } 
        = S_\alpha\paraa{ S_\alpha(f_1,f_2,f_3), f_4, f_5}
  \end{align}
  where $S_\alpha$ with three arguments is the symbol (\ref{tri_sym}) of
  the triple operator.
\end{proposition}
\begin{proof}
  We write the symbol (\ref{tri_sym}) as integral
  \begin{align} \label{tri_sym_short}
    S_\alpha\paraa{f_3, f_4,f_5}(x,\zb) =  \frac{\alpha^N}{\pi^N} \int  
       \paraa{e^{-\alpha \tilde{\phi}}\tilde{\mu}}^{ x \zb  y \yb} g^{y\yb}
        f_3^{x\yb}f_4^{y\yb}f_5^{y\zb} \, d^{2N}y 
  \end{align} 
   for $(x,\zb)$ in a neighbourhood of the diagonal, which
   follows directly from (\ref{int_exp}) and (\ref{J_int}). 
   To shorten notation, we have written the arguments as upper indices.
 
  The symbol of $\hat{T} = T_{f_1}M_{f_2} \paraa{ T_{f_3} M_{f_4} T_{f_5} }$
  written as double integral is
  \begin{align} \label{trip_ls}
     \frac{\alpha^{2N}}{\pi^{2N}} 
       \int \paraa{e^{-\alpha \tilde{\phi}}\tilde{\mu}}^{ x \zb  y \yb} g^{y\yb}
        f_1^{x\yb}f_2^{y\yb}  
        \parab{ \int  \paraa{e^{-\alpha \tilde{\phi}}\tilde{\mu}}^{ y \zb  w \wb} g^{w\wb}
          f_3^{y\wb}f_4^{w\wb}f_5^{w\zb} \, d^{2N}w   }
     \, d^{2N}y 
  \end{align} 
  The symbol of $\hat{T} = \paraa{ T_{f_1}M_{f_2} T_{f_3} } M_{f_4} T_{f_5}$
  written as double integral is
  \begin{align} \label{trip_rs}
     \frac{\alpha^{2N}}{\pi^{2N}} 
       \int \paraa{e^{-\alpha \tilde{\phi}}\tilde{\mu}}^{ x \zb  w \wb} g^{w\wb}
        \parab{ \int  \paraa{e^{-\alpha \tilde{\phi}}\tilde{\mu}}^{ x \wb  y \yb} g^{y\yb}
          f_1^{x\yb}f_2^{y\yb}f_3^{y\wb} \, d^{2N}y   }
        f_4^{w\wb}f_5^{w\zb}  
     \, d^{2N}w
  \end{align}
  We note that
  \begin{align*} 
      \paraa{e^{-\alpha \tilde{\phi}}\tilde{\mu}}^{ x \zb  y \yb}
      \paraa{e^{-\alpha \tilde{\phi}}\tilde{\mu}}^{ y \zb  w \wb} =
      \paraa{e^{-\alpha \tilde{\phi}}\tilde{\mu}}^{ x \zb  w \wb}
      \paraa{e^{-\alpha \tilde{\phi}}\tilde{\mu}}^{ x \wb  y \yb}=
     \paraa{e^{-\alpha \check{\phi}}}^{ x \zb  y \yb w \wb}
     \tilde{\mu}^{ x \zb  y \yb}  \tilde{\mu}^{ y \zb  w \wb} 
  \end{align*}  
  which follows by inserting the definitions (\ref{dia_fun_phi}), (\ref{dia_fun_mu}) of $\tilde{\phi}$
  and $\tilde{\mu}$. Therefore, both integrals (\ref{trip_ls}) and (\ref{trip_rs}) 
  exist due to Lemma \ref{cphi_lem}, are equal and have an asymptotic
  expansion $\tilde{S}_\alpha$ in a neighbourhood of the diagonal. This asymptotic
  expansion is the analytic symbol of $\hat{T}$.
  (\ref{5_sym}) follows by comparing (\ref{trip_ls}) and (\ref{trip_rs})  with (\ref{tri_sym_short}).
\end{proof}

To arrive at a closed algebra of Töplitz operators, one has to
show that also the triple operator (\ref{trip_prod}) of three
analytic symbol (and not only principal symbols) is a covariant Töplitz
operator. This is done for the case $\tilde{\mu}=1$
in \cite{Cha03, Del21,Hez21,Cha21}, where it turns out that the definition of
an analytic symbol has to be further restricted.

We will however go in another direction and simply define a formal triple symbol
in analogy with (\ref{tri_sym}). Nevertheless, under the assumption that the corresponding
Töplitz operators and asymptotic expansions converge, the formulas stay the
same by simply replacing $\hbar$ with $\frac{1}{\alpha}$.

\section{Formal symbols and their $\star$-products}

\begin{definition}  
  Let $f_1$, $f_2$, $f_3$ be formal symbols.
  The \emph{formal triple symbol} is 
\begin{align} \label{tri_sym}
  S\paraa{f_1, f_2,f_3}(x,\zb)  = \sum_{n,m,p,q=0}^{\infty} \hbar^{n+m+p+q}
       R_n\parab{ f_{1,m}^{x\yb} f_{2,p}^{y\yb}f_{3,q}^{y\zb} \tilde{\mu}^{x\zb y\yb} }(x,\zb)
\end{align}
\end{definition}  
The formal triple symbol is the analytic symbol (\ref{tri_sym})
of the triple operator (\ref{trip_prod}), where the
asymptotic expansion is replaced by the formal power series in $\hbar$.

In \cite{Eng00-2}  it is shown that the coefficients $R_n$ of the expansion (\ref{int_exp}) solely
depend on the Riemannian curvature of  the metric $g_{i\jb}=\partial_i\partial_{\jb}\phi$
of the Kähler potential $\phi$, its contractions
and covariant derivatives thereof. In particular,
\begin{align}
  R_0\paraa{f}(x,\zb)  &=  f(x,\zb,x,\zb) \label{R0_eq} \\
  R_1\paraa{f}(x,\zb)  & = \paraa{\Delta f + \frac{1}{2}R f}(x,\zb,x,\zb) \label{R1_eq}
\end{align} 
for a function $f(x,\zb,y,\yb)$, where $R=g^{\jb i}\partial_i \partial_{\jb}(\ln g)$ is the scalar curvature and
$\Delta=g^{i\jb}\partial_i \partial_{\jb}$ is the Laplace operator. The derivatives
are with respect to $y$ and $\yb$.
In \cite{Eng00-2} also $R_2$ (Theorem 5) and $R_3$ (page 23) are calculated.

Up to first order, the formal triple symbol is explicitly
\begin{align} \label{tri_sym_form}
  S\paraa{f_1, f_2,f_3}  
      & = R_0 \paraa{  f_1^{x\yb} f_2^{y\yb}f_3^{y\zb} \tilde{\mu}^{x\zb y\yb} } 
            +  \hbar R_1 \parab{  f_1^{x\yb} f_2^{y\yb}f_3^{y\zb} \tilde{\mu}^{x\zb y\yb} }
            + \mathcal{O}(\hbar^2) \\
      & = f_1 f_2 f_3\paraa{1 + \hbar (\Delta \mu + \frac{1}{2} R) } \nonumber \\
       &    + \hbar g^{i\jb} \paraa {f_1 f_3 f_{2,i\jb} + f_3 f_{1,\jb}f_{2,i} + f_2 f_{1,\jb}f_{3,i} + f_1 f_{2,\jb}f_{3,i} }
            +  \mathcal{O}(\hbar^2)    \nonumber  
\end{align}
where all functions are evaluated at $(x,\zb)$ and,
we have used that $\tilde{\mu}=1$,
$\tilde{\mu}_{,I}=0$ and 
$\tilde{\mu}_{,i\jb}=\mu_{,i\jb}$
at $(x,\zb,y,\yb)=(x,\zb,x,\zb)$. We have
abbreviated partial derivatives with indices separated by a comma.
(Remember that the formal symbols $f_1$, $f_2$ and $f_3$ 
are additional formal power series in $\hbar$.)

When we replace the asymptotic expansion of 
Proposition \ref{5_sym_prop} with a formal power series, we also
get a generalized associativity law for the formal triple symbol.
\begin{proposition} \label{S_ass_prop}
    Let $f_1$, $f_2$, $f_3$, $f_4$ and $f_5$ be formal symbols. Then
\begin{align} \label{S_ass}
  S\paraa{f_1, f_2,  S\paraa{f_3, f_4,f_5} } = S\paraa{S\paraa{f_1, f_2,f_3}, f_4, f_5 }
\end{align}
\end{proposition}
\begin{proof}
  When $f_1$, $f_2$, $f_3$, $f_4$ and $f_5$ are primary symbols, 
  Proposition \ref{5_sym_prop} tells that (\ref{5_sym})
  are equal asymptotic expansions, in which we
  can consider every order in $\frac{1}{\alpha}$ and
  replace $\frac{1}{\alpha}$ with $\hbar$.
  It follows that (\ref{S_ass}) is true for primary symbols.
  The general case follows by replacing the primary symbols
  with formal symbols, since for formal power series, the sums in $\hbar$ can be exchanged.
\end{proof}

The rest of this section is devoted to proof Theorem \ref{main_theo}.

\begin{proof}

The Berezin-Töplitz $\star$-product is defined by
$f\star h = S(f,1,h)$ for two formal symbols $f$ and $h$.
Equation (\ref{ber_top_exp}) follows by inserting this
definition into (\ref{tri_sym}). Equation (\ref{ber_top_star_h2})
can be directly derived from (\ref{tri_sym_form}).

The Berezin-Töplitz-$\star$-product (\ref{ber_top_exp}) is an associative product,
since by Proposition \ref{S_ass_prop}
\begin{align}
  f_1 \star (f_2\star f_3) = S\paraa{f_1, 1,  S\paraa{f_2, 1,f_3} }
      = S\paraa{S\paraa{f_1, 1,f_2}, 1, f_3 } = (f_1 \star f_2 ) \star f_3
\end{align}
for three formal symbols $f_1$, $f_2$ and $f_3$.

When $k$ is the unit of the Berezin-Töplitz $\star$-product,
it follows that $k \star 1=1$ and a linear recurrence relation for $k$ can be derived. 
\begin{align*}
  1  & = \sum_{n,m=0}^{\infty} \hbar^{n+m}  R_n \paraa{ k_m^{ x \yb}  \tilde{\mu}^{x\xb y\yb} } 
      = \sum_{n=0}^{\infty} \hbar^n \sum_{m=0}^n  R_m \paraa{ k_{n-m}^{x \yb}  \tilde{\mu}^{x\xb y\yb} }  \\
     & = k_0 +  \sum_{n=1}^{\infty} \hbar^n \parab{ k_{n}
              +   \sum_{m=1}^{n}  R_m \paraa{ k_{n-m}^{x \yb}  \tilde{\mu}^{x \xb y\yb} }  }
\end{align*}
( see \cite{Cha03}, \cite{Hez21})
Thus,
\begin{align*} 
   k_0 = 1, \qquad k_n  = - \sum_{m=1}^{n} R_m \paraa{ k_{n-m}^{x \yb}  \tilde{\mu}^{x\xb y\yb} }
\end{align*}
Therefore, the formal symbol $k$ of the Bergman kernel exists and can be
calculated by solving the recurrrence relation order by order.
In particular, by applying (\ref{R0_eq}) and (\ref{R1_eq}) equation (\ref{ker_sym_h2}) follows.

As a further option, one can consider $k\star k = k$
from which a quadratic recurrence relation can be derived (see \cite{Eng00-2}).
\begin{align*}
  \sum_{n=0}^{\infty} \frac{1}{\alpha^n}k_n 
    & = \sum_{n,m,p=0}^{\infty} \frac{1}{\alpha^{n+m+p} } R_n \paraa{ k_m^{\yb} k_n^y \tilde{\mu}^{y\yb} } 
      = \sum_{n=0}^{\infty} \frac{1}{\alpha^n} \sum_{m=0}^n  R_m \paraa{ \sum_{p=0}^{n-m} k_p k_{n-p}  \tilde{\mu}^{y\yb} }  \\
    & = k_0^2 + \sum_{n=1}^{\infty} \frac{1}{\alpha^n} \parab{
      2 k_0 k_{n}  + \sum_{p=1}^{n-1} k_p k_{n-p}   
     +  \sum_{m=1}^n  R_m \paraa{ \sum_{p=0}^{n-m} k_p k_{n-p}  \tilde{\mu}^{y\yb} }    }
\end{align*}
Thus,
\begin{align*} 
   k_0 = 1, \qquad
   k_n  = - \sum_{p=1}^{n-1} k_p k_{n-p}  -  \sum_{m=1}^n  R_m \paraa{ \sum_{p=0}^{n-m} k_p k_{n-p}  \tilde{\mu}^{y\yb} }
\end{align*}

Using the definition (\ref{tri_sym}) of the triple symbol $S$, 
the Berezin transform ( \ref{ber_trans}) is
\begin{align}
   \psi(f)(x,\xb) =  \sum_{n=0}^{\infty} \hbar^n R_n\paraa{ k^{x \yb}f^{y\yb} k^{y\xb} \tilde{\mu}^{x\xb y\yb} }(x,\xb) 
\end{align}
Expanding $k$ results in
\begin{align*}  
   \psi(f) & = \sum_{n,m,p=0}^{\infty} \hbar^{n+m+p}  R_n \paraa{ k_m^{x \yb} f^{y\yb} k_p^{y\xb}  \tilde{\mu}^{x\xb y\yb} }
\end{align*}
and
\begin{align*}  
  \psi(f)_0 & = f \\
  \psi(f)_1 & = 2 k_1  f + R_1 \paraa{ (f^{y\yb} \tilde{\mu}^{x\xb y\yb} } \\
      & = -  2 \paraa{\Delta \mu + \frac{1}{2} R}  f  + \Delta f + f \Delta \mu + \frac{1}{2} f R \\
      & = \Delta f  -  f \paraa{\Delta \mu + \frac{1}{2} R}
\end{align*}
In summary,
\begin{align}  \label{ber_trans_h2}
  \psi(f) = f +\hbar \paraa{ \Delta f  -  f \paraa{\Delta \mu + \frac{1}{2} R} } +  \mathcal{O}(\hbar^2)    
\end{align}

To calculate $\psi^{-1}(f)$ assume that
\begin{align*}  
   \psi^{-1}(f) = \sum_n \hbar^n g_n
\end{align*}  
and therefore
\begin{align*}  
   f  = \sum_{n,m,p,q=0}^{\infty} \hbar^{n+m+p+q} R_n \paraa{ k_m^{x\yb} k_p^{y\xb} g_q^{y\yb} \tilde{\mu}^{x\xb y\yb} } 
       = \sum_{n=0}^{\infty} \hbar^n \sum_{l+m+p+q=n}  R_l \paraa{ k_m^{x\yb} k_p^{y\xb} (g_q^{y\yb} \tilde{\mu})^{ x\xb y\yb} }  
\end{align*}
which results in a recurrence relation for the $g_n$
\begin{align*}  
     g_0 = f, \qquad g_n  = - \sum_{r=0}^{n-1} \sum_{m+p+q=n-r}  R_m \paraa{ k_p^{x\yb} k_q^{y\xb} (g_{r}^{y\yb} \tilde{\mu})^{x\xb y\yb} }  
\end{align*}
The second relation is
\begin{align*}  
     g_1 & =  -  2 k_1 g_0  - R_1 \paraa{ (g_0^{y\yb} \tilde{\mu})^{x\xb y\yb} } \\
      & =   2 (\Delta \mu + \frac{1}{2} R) f  - \paraa{ \Delta f +f\Delta\mu  + \frac{1}{2} f_0  R  } \\
      & =  - \Delta f  + f (\Delta \mu + \frac{1}{2} R) 
\end{align*}
In summary,
\begin{align}  \label{inv_ber_trans_h2}
   \psi^{-1}(f) = f - \hbar \paraa{ \Delta f  -  f \paraa{\Delta \mu + \frac{1}{2} R} } +  \mathcal{O}(\hbar^2)    
\end{align}

To determine the contravariant $\star$-product (\ref{con_star})
we calculate up to first order
\begin{align*} 
 \psi(f) \star \psi(h) = fh + \hbar( h \Delta f + f \Delta h + g^{i\jb} f_{,\jb} h_{,i} -  fh \paraa{\Delta \mu + \frac{1}{2} R})
     +  \mathcal{O}(\hbar^2)    
\end{align*}
where we have used (\ref{ber_top_star_h2}) and (\ref{ber_trans_h2}).
(\ref{con_star_h2}) follows by additionally applying (\ref{inv_ber_trans_h2}).

For the covariant $\star$-product (\ref{cov_star}), one determines the formal
inverse of $k$ and applies this to the product (\ref{ber_top_star_h2}) of $kf$ and
$kh$. This results in (\ref{cov_star_h2}).
\end{proof}

%\bibliographystyle{alpha}
%\bibliography{references}  

\end{document}